\documentclass[12pt,leqno,oneside]{amsart}
\usepackage{amssymb,amsmath,amscd,graphicx,fontenc,amsthm,mathrsfs, mathtools}
\usepackage[frame, cmtip, arrow, matrix, line, graph, curve]{xy}
\usepackage{graphpap, color}
\usepackage[mathscr]{eucal}
\usepackage{ifthen}
\usepackage{relsize}
\newcommand{\vertiii}[1]{{\left\vert\kern-0.25ex\left\vert\kern-0.25ex\left\vert #1
    \right\vert\kern-0.25ex\right\vert\kern-0.25ex\right\vert}}
\theoremstyle{plain}
\usepackage{geometry}
\usepackage[colorlinks = true,
            linkcolor = blue,
            urlcolor  = blue,
            citecolor = blue,
            anchorcolor = blue]{hyperref}
\usepackage{fancyhdr}
\usepackage{indentfirst}
\usepackage{paralist}
\usepackage{pstricks}
\usepackage{bbm}
\usepackage{tikz}
\usepackage{enumitem}

\newtheorem{theorem}{Theorem}[section]

\newtheorem{lemma}[theorem]{Lemma}

\newtheorem{corollary}[theorem]{Corollary}

\theoremstyle{remark}

\numberwithin{equation}{section}

\newcommand{\NN}{{\mathbb{N}}}
\newcommand{\ZZ}{{\mathbb{Z}}}

\newcommand{\RR}{{\mathbb{R}}}

\renewcommand{\SS}{{\mathbb{S}}}

\DeclareMathOperator{\vol}{vol}

\DeclareMathOperator{\supp}{supp}

\newcommand{\GL}{\mathrm{GL}}
\newcommand{\SL}{\mathrm{SL}}

\newcommand{\SO}{{\mathrm{SO}}}

\newcommand{\IST}{\mathrm{I}}

\providecommand{\M}{M}  
\providecommand{\lin}{l} 

\providecommand{\vi}{i_0,\ldots,i_{r}}

\providecommand{\rchi}{\raisebox{2pt}{$\chi$}}
\providecommand{\Error}[2]{\mathcal E^{#1}_{#2}}

\makeatletter
\def\moverlay{\mathpalette\mov@rlay}
\def\mov@rlay#1#2{\leavevmode\vtop{%
   \baselineskip\z@skip \lineskiplimit-\maxdimen
   \ialign{\hfil$\m@th#1##$\hfil\cr#2\crcr}}}
\newcommand{\charfusion}[3][\mathord]{
    #1{\ifx#1\mathop\vphantom{#2}\fi
        \mathpalette\mov@rlay{#2\cr#3}
      }
    \ifx#1\mathop\expandafter\displaylimits\fi}
\makeatother

\definecolor{cmd}{rgb}{1.0, 0.35, 0.21}

\begin{document}
\title{A generic effective Oppenheim theorem for systems of forms}
\author{Prasuna Bandi}
\author{Anish Ghosh}
\author{Jiyoung Han}
\address{PB and AG: School of Mathematics, Tata Institute of Fundamental Research, Mumbai, 400005, India}
\address{JH: Research Institute of Mathematics, Seoul National University, Seoul, 08826, South Korea}
\email{prasuna@math.tifr.res.in, ghosh@math.tifr.res.in, jiyoung.han.math@snu.ac.kr}
\thanks{AG was supported by a Government of India, Department of Science and Technology, Swarnajayanti fellowship DST/SJF/MSA-01/2016--17, a grant from the Infosys foundation, a CEFIPRA grant and a Matrics grant. PB and AG acknowledge support of the Department of Atomic Energy, Government of India, under project $12-R\&D-TFR-5.01-0500$. JY is supported by the Samsung Science and Technology Foundation under project No. SSTF-BA1601-03 and the National Research Foundation of Korea(NRF) grant funded by the Korea government under project No. 0409-20200150.}


\maketitle
\begin{abstract}
We prove a uniform effective density theorem as well as an effective counting result for a generic system comprising \textcolor{black}{a polynomial with a mild homogeneous condition} and several linear forms using \textcolor{black}{Rogers'} second moment formula for the Siegel transform on the space of unimodular lattices.
\end{abstract}

\section{Introduction}
In this paper, we investigate the effective density of values of a system of forms at integer points. We recall Margulis's famous result \cite{Ma89}, resolving an old conjecture of Oppenheim: if $Q$ is an indefinite nondegenerate quadratic form in at least $3$ variables which is not proportional to a form with integer coefficients, then $\{Q(x) : x \in \ZZ^n\}$ is dense in $\RR$. Recently, there has been a surge of interest in \emph{effective} versions of Margulis's result. The basic question in this area is: given $\xi \in \RR$ and $\epsilon > 0$, how large must $x \in \ZZ^{n}$ be so that
$$|Q(x) - \xi| < \epsilon?$$
Margulis's proof of his theorem is based on dynamics on the space of unimodular lattices, and does not easily lend to effectivising. Indeed, effective results constitute one of the main current challenges in homogeneous dynamics. In \cite{LM14}, Lindenstrauss and Margulis investigated this problem for ternary quadratic forms and found a logarithmic in $\epsilon$ bound for $x$ as above, for a large class of quadratic forms satisfying an explicit Diophantine condition. In \cite{GGN18}, Ghosh, Gorodnik and Nevo  showed that one can do much better for a generic form. Namely, \textcolor{black}{for any given $\epsilon>0$,} it was shown that for almost every quadratic form in $3$ variables, and for every $\eta > 1$, 
\begin{equation}\label{eq:intro}
|Q(x) - \xi| < \epsilon \text{ and } \|x\| \leq \frac{1}{\epsilon^{\eta}}
\end{equation}
admits a solution. \textcolor{black}{Here, $\|\cdot\|$ is the maximum norm.} It can be shown that the exponent $1$ in (\ref{eq:intro}) is sharp. The method of proof in \cite{GGN18} involves effective mean ergodic theorems and duality techniques and applies to a wide variety of Diophantine problems. However as far as the classical Oppenheim problem is concerned, while this technique applies to quadratic forms in any number of variables, it gives the best possible result only in dimension $3$; the quality of the exponent deteriorates as the dimension increases. In \cite{AM18}, Athreya and Margulis used a different approach; they used Rogers second moment formula in the space of lattices to obtain the right exponent in (\ref{eq:intro}) in all dimensions for the special case $\xi = 0$.   

Another natural problem in this setting is the quantitative form of the Oppenheim conjecture. Namely, given a quadratic form $Q$ and an interval $I$, one seeks to study the counting function
\begin{equation}\label{eq:intro2}
\mathcal{N}_{Q,I}(t)=\#\{x\in\mathbb{Z}^n :  Q(x)\in I,\ \|x\|\leq t\}.
\end{equation}

Explicit asymptotics for the counting function have been obtained in \cite{DM93} (lower bounds) and \cite{EMM98, EMM05} (upper bounds) for every nondegenerate, indefinite, irrational quadratic form. Namely, it is known that\footnote{\textcolor{black}{The} upper bounds are more complicated for forms of signature $(2, 1)$ and $(2, 2)$; we refer the reader to \cite{EMM05}.} 
$$\mathcal{N}_{Q,I}(t) \approx_Q |I| t^{n-2}. $$

In \cite{AM18}, this quantitative result was sharpened to obtain an error term for generic forms. Namely, it was shown that there exists $\nu > 0$ such that for every interval $I$ and for almost every quadratic form Q,
$$\mathcal{N}_{Q,I}(t) = c_Q |I| t^{n-2} + O_{Q, I}(t^{n-2-\nu}). $$
\textcolor{black}{Here, $c_Q>0$ is a constant depending only on the quadratic form $Q$ (For details, see the remark below Lemma 3.8 in \cite{EMM98}).} 

The results in \cite{AM18} were generalised by Kelmer and Yu \cite{KY18} in three different regimes: they allowed the intervals in (\ref{eq:intro2}) to shrink; they considered more general homogeneous polynomials and confirmed a prediction of Ghosh, Gorodnik and Nevo \cite{GGN18}; and they considered \emph{uniform} versions of these results, in other words, they considered the situation where a single random quadratic form approximates all the points $\xi$. Such uniform results were first considered by Bourgain \cite{Bou16} for diagonal ternary forms and then by Ghosh and Kelmer in \cite{GK18} for general ternary forms. The method of proof in \cite{KY18} also relies on Rogers' second moment formula. We refer the reader to \cite{BGHM10, GK17, GKY19, GKY20, KS19} for other recent works on effective versions of the Oppenheim conjecture in various contexts.

\subsection{Systems of forms}
Much less is known when one considers the natural generalisation of Oppenheim type problems to systems of forms. This problem was first considered by Dani and Margulis in \cite{DM90} who gave sufficient conditions for the density of a pair $(Q, L)$ \textcolor{black}{of a quadratic form and a linear form} in $3$ variables. This was generalised by Gorodnik \cite{Gor04} to pairs $(Q, L)$ in four or more variables. Further work on systems of forms has been done by Gorodnik \cite{Gor04b} for systems of quadratic forms, by Dani \cite{Dan00, Dan08} for systems comprising a quadratic form and a linear form, and by M\"{u}ller \cite{Mu05, Mu08} for certain systems of quadratic forms. Other than the papers of M\"{u}ller which use the circle method and therefore get quantitative results, the other works use homogeneous dynamics and establish qualitative statements. In particular, the general problem of establishing quantitative and effective versions of Oppenheim type conjectures for systems of forms seems wide open. In this paper, we will prove counting results with error term, as well as effective theorems for a generic system comprising a quadratic form (more generally, a homogeneous polynomial) and a system of linear forms.  Following \cite{AM18, KY18} we will use Rogers' formula; in fact we follow the strategy of Kelmer and Yu \cite{KY18} closely. The main new ingredient in this present paper is a volume calculation, Theorem \ref{volume form}. We note that the problems considered in this paper do not seem to be amenable to the ergodic approach of \cite{GGN18}, which requires semisimple stabilizers. However, a similar problem, that of the effective density of linear maps taking values on rational quadratic surfaces can be addressed using ergodic methods, see Theorem 1.5 in \cite{GGN18}. Previously density and counting results in this setting were proved by Sargent \cite{Sar14, Sar14b}. 

\subsection{Main results}
In order to state our main results, we need to recall a classification of the systems of forms from \cite{Sar14}.
\subsubsection{Classification}\label{classification_old}

Consider the space of systems $(Q, \M)$ of a nondegenerate quadratic form $Q$ on $\RR^n$ and a linear map $\M : \RR^n \rightarrow \RR^r$ of rank $r$, for some given $r < n$. 

We define an equivalence relation on the space of systems $(Q,M)$ as follows:  $(Q_1,\M_1)\sim (Q_2,\M_2)$ 
if there exist $(g_1, g_2)\in \SL_n(\RR)\times \GL_r(\RR)$ and $0\neq\lambda\in \RR$ 
such that \textcolor{black}{$(Q_2, \M_2)=(\lambda Q^{g_1}_1, g_1\M_1 g_2)$}. Here, $Q^{g}(v) :=  Q(vg)$. \\
Define
\[
	\mathcal{Y}^{(p,q,u,v)}:=\left\{(Q,M): \begin{split}&\text{Q is a nondegenerate quadratic form on } \RR^{n}\text{ with sign}(Q)=(u,v),\\
	&M:\RR^{n}\rightarrow \RR^{r}\text{ is a linear map of rank r, and sign}(Q|_{\ker(M)})=(p,q). 
	\end{split}\right\}
\]

According to \cite{Sar14}, which in turn is adapted from \cite{Gor04}, $(Q,M)\in \mathcal{Y}^{(p,q,u,v)}$ is equivalent to $(Q_0, \M_0)$, where

\begin{equation*}
\left\{\begin{array}{l}
Q_0(x_1, \ldots, x_{p+q}, y_1, \ldots, y_{2t}, z_1, \ldots, z_s)\\
\hspace{0.2in}=(x_1^2+\cdots+x_p^2-x_{p+1}^2-\cdots-x_{p+q}^2)+(2y_1y_{t+1}+\cdots+2y_ty_{2t})\\
\hspace{0.28in}+Q'(z_1, \ldots, z_s);\\[0.05in]
\M_0(x_1, \ldots, x_{p+q}, y_1, \ldots, y_{2t}, z_1, \ldots, z_s)
=(y_{t+1}, \ldots, y_{2t}, z_1, \ldots, z_s),
\end{array}\right.
\end{equation*}
and $Q'(z_1, \ldots, z_s)=z_1^2+\cdots+z_{p'}^2-z_{p'+1}^2-\cdots-z_{p'+q'}^2$.
Here $t=n-r-(p+q),\; p'=u-t-p,\; q'=v-t-q \text{ and } s=p'+q'$.

 Hence one can identify $\mathcal Y^{(p,q,u,v)}$ with $(\RR-\{0\})\times \SL_n(\RR) \times \GL_r(\RR)/\SO(Q_0,\M_0)$, where $\SO(Q_0,\M_0)$ is the isotropy subgroup of $(Q_0,\M_0)$ in $(\RR-\{0\})\times\SL_n(\RR) \times \GL_r(\RR)$. Let us assign the measure on $\mathcal Y^{(p,q,u,v)}$ induced from the Haar measure on $(\RR-\{0\})\times \SL_n(\RR) \times \GL_r(\RR)$, which is the product of the Lebesgue measure on $\RR$ and Haar measures on $\SL_n(\RR)$ and $\GL_r(\RR)$, respectively.

\subsubsection{The homogeneous space $\mathcal Y^{(F_0, \M_0)}$}\label{sec:classification}
More generally, let $(F_0, M_0)$ be a system of \textcolor{black}{a polynomial in $n$ variables} and a linear map $M_0 : \RR^n \rightarrow \RR^r$ of the form:
\begin{equation}\label{eqn:def}
\left\{\begin{array}{l}
F_0(x_1, \ldots, x_{p+q}, y_1, \ldots, y_{2t}, z_1, \ldots, z_s)\\
\hspace{0.2in}=(x_1^d+\cdots+x_p^d-x_{p+1}^d-\cdots-x_{p+q}^d)+P_1(y_1, \ldots, y_{2t})\\
\hspace{0.28in}+P_2(z_1, \ldots, z_s);\\[0.05in]
\M_0(x_1, \ldots, x_{p+q}, y_1, \ldots, y_{2t}, z_1, \ldots, z_s)
=(y_{t+1}, \ldots, y_{2t}, z_1, \ldots, z_s),
\end{array}\right.
\end{equation}
where \textcolor{black}{$d\ge 2$ is even,} $P_1(y_1, \ldots, y_{2t})$ is a polynomial such that there is a positive integer $d'<d$ for which $P_1(Ty_1, \ldots, Ty_t, y_{t+1},\ldots, y_{2t})=O(T^{d'})$ as $T$ goes to infinity, and $P_2(z_1, \ldots, z_s)$ is any polynomial in the variables $z_1, \ldots, z_s$. 

Define $\mathcal Y^{(F_0,\M_0)}$ by
\[
\mathcal Y^{(F_0, \M_0)}:=
\left\{(F, M)= (\lambda F_0^{g_1}, g_1\M_0g_2) :
(\lambda,g_1, g_2) \in (\RR-\{0\})\times\SL_n(\RR)\times\in \GL_r(\RR) \right\}
\]
so that as in Section \ref{classification_old}, 
we can identify $\mathcal Y^{(F_0, \M_0)}$ with the symmetric space $(\RR-\{0\})\times \SL_n(\RR)\times \GL_r(\RR)/\IST_{(F_0,\M_0)}$, 
where $\IST_{(F_0,\M_0)}$ is the isotropy subgroup. 
\textcolor{black}{Using} this identification, we will assign the $(\RR-\{0\})\times \SL_n(\RR)\times \GL_r(\RR)$-invariant measure on $\mathcal Y^{(F_0, \M_0)}$.

\textcolor{black}{We will not specify the norm $\|\cdot\|$ on $\RR^n$ at present. If we need to fix a norm, we will specify it in the relevant statement.} 

\vspace{0.2in}
Our first theorem proves effective counting for generic forms with values in possibly shrinking sets. It is an analogue, for systems of forms, of Theorem 1 in \cite{KY18}.

\begin{theorem}\label{thm:main1_new}
		Let $p \ge 1$, $q \ge 1$ with $d+1\le p+q\le n-r$ and \textcolor{black}{let} $0\le \kappa<n-r-d$. Let $\{I_{t}\}_{t>0}$ be a \textcolor{black}{non-increasing} family of bounded measurable subsets of $\RR^{r+1}$ with $|I_{t}|=\textcolor{black}{ct^{-\kappa}}$ for some $c>0$. Then there is $\nu >0$ such that for almost every $(F,M)\in \mathcal{F}^{(F_0,M_0)}$, there exists $c_{F,M}>0$ such that 
		\[
		\#\{v\in \ZZ^{n} : (F,M)(v)\in I_{t}, \|v\|\le t\}=c_{F,M}|I_{t}|t^{n-r-d}+O_{F,M}(t^{n-r-d-\kappa-\nu}).
		\]
	Here, $|I|$ is the Lebesgue measure of a subset $I \subset \RR^{r+1}$.
	\end{theorem}
	
In particular, we obtain the following corollary which constitutes an effective version of Oppenheim's conjecture for systems of forms.

\begin{corollary}\label{cor:main1_new}
	Let $p \ge 1$, $q \ge 1$ with $d+1\le p+q\le n-r$ and $0\le \kappa<n-r-d$. 
	Let $(\kappa_{0},\kappa_{1},\ldots,\kappa_{r})\in \RR_{>0}^{r+1}$ be such that $\kappa_{0}+\kappa_{1}+\cdots+\kappa_{r}=\kappa$. 
	Then for any $\xi=(\xi_0,\xi_1,\ldots,\xi_{r})\in \RR^{r+1}$ and for almost every $(F,M)\in \mathcal{F}^{(F_0,M_0)}$, the system of inequalities
\[\left\{\begin{array}{l}
|F(v)-\xi_0|<t^{-\kappa_0};\\
|\lin_i(v)-\xi_{i}|<t^{-\kappa_{i}},\;1 \le i \le r;\\
\|v\|\le t, \end{array}
\right.\]
where $M=(\lin_1,\ldots,\lin_r)$, has integer solutions for sufficiently large t.
\end{corollary}	

Our second main theorem is a `uniform' effective counting result.  
\begin{theorem}\label{thm:main2_new}
Let $p \ge 1$, $q \ge 1$ with $d+1\le p+q\le n-r$. 
Let $0\leq \eta < \min\{1,\frac{n-r-d}{(r+1)(1+r(r+2))}\}$ and $0\le\kappa< \frac{n-r-d-(r+1)\eta-r(r+1)(r+2)\eta}{(r+1)(r+2)}$. 
For $0\le j\le r$, let $\kappa_j >0$ be such that $\sum_{j=0}^{r}\kappa_j=\kappa$. 
Let $N(t)$ be a non-decreasing function such that $N(t)=O(t^{\eta})$. 
Then there exists $\nu>0 $ such that for almost every $(F,M)\in \mathcal{F}^{(F_0,M_0)}$ and for \textcolor{black}{all $I\subset [-N(t),N(t)]^{r+1}$} of the form $I=I_0\times I_1\times \cdots \times I_{r}$ with $|I_j|\ge t^{-\kappa_j}$, 
\[
\#\{v\in \ZZ^{n} : (F,\M)(v)\in I, \|v\|\le t\}=c_{F,M} |I |t^{n-r-d}+O_{F,M}(|I|t^{n-r-d-\nu}).
\]
\end{theorem}

This \textcolor{black}{theorem} implies Corollary \ref{cor1.4}, which is precisely the uniform version of effective Oppenheim studied in \cite{Bou16} and \cite{GK17}.
\begin{corollary}\label{cor1.4}
	Let \textcolor{black}{$d < n-r$} and $0\leq \eta < \min\{1,\frac{n-r-d}{(r+1)(1+r(r+2))}\}$. Let $N(t)$ be a non-decreasing function such that $N(t)=O(t^{\eta})$ and $\delta(t)$ be a non-increasing function satisfying $\frac{t^{\eta(r+1)(1+r(r+2))-a}}{\delta(t)^{(r+1)^{2}(r+2)}}\rightarrow 0$ for some $a<n-r-d$. Then for almost every $(F,M)\in\mathcal{Y}^{(F_0,M_0)}$ and for sufficiently large $t$,
	\[
	\sup_{\|\xi\|\le N(t)} \min_{v\in \ZZ^{n},\|v\|\le t}\|(F,M)(v)-\xi\|<\delta(t),
	\]
	where $\|\cdot\|$ denotes the supremum norm on $\RR^{n}$.
\end{corollary}

Note that by replacing $(F_{0},M_{0})$ with $(Q_{0},M_{0})$ and taking $d=2$, one can deduce that the above stated theorems hold for almost every $(Q,M)\in \mathcal{Y}^{(p,q,u,v)}$.

\section{Volume Estimation}
Following \cite{KY18}, we estimate the volume of the region given as the preimage of a system of \textcolor{black}{a polynomial $F_0$} and a linear map $M_0$ defined as in \eqref{eqn:def}. 
For simplicity, let us denote $Y_1=(y_1, \ldots, y_t)$, $Y_2=(y_{t+1}, \ldots, y_{2t})$ and $Z=(z_1, \ldots, z_s)$.

Recall that for two functions $f(T)$ and $f'(T)$, we denote that $f(T)\ll f'(T)$ if there is some constant $c>0$ such that $f(T)<cf'(T)$ for all (sufficiently large) $T>0$. 
We will use the notation $\ll_a$ when we want to specify that a constant $c$ depends on a variable $a$.

\begin{lemma}\label{smooth version} 
\textcolor{black}{Let $(F_0, M_0)$ be as in \eqref{eqn:def}.} Let $p\ge 1,\; q\ge 1 \text{ with }d+1\le p+q\le n-r \text{ and let }$ $h : \RR^n \rightarrow \RR$ be a compactly supported smooth function such that $\supp h \subseteq B_a(0) \subset \RR^n$ for some $a>1$, \textcolor{black}{where $B_a(0)$ is the ball of radius $a$ centered at the origin with respect to an arbitrary norm on $\RR^n$.} 
Let $I\subseteq [-1,1]^{r+1}\subset \RR^{r+1}$ be a measurable set.
For $\alpha\in (0,\frac {d-d'} d)$, there is $T_0>0$ such that for $T>T_0$,
\begin{equation}\label{eq:smooth:3}\begin{split}
\mathcal I_{h,I}
&:=\int_{\RR^n} h\Big(\frac v T\Big)\rchi_{I}(F_0(v),\M_0(v)) dv\\
&=J(h) |I| T^{n-r-d}\\
&\hspace{0.2in}+\left\{\begin{array}{lc}
\begin{array}{l}
O_a\left(\mathcal S_1(h)|I|T^{p+q+t-d-1}\right)\\
\hspace{0.5in}+O_a\left(\|h\|_\infty|I|T^{(1-\alpha)(p+q-d)+t}\right),
\end{array}
&\text{if }p+q \ge 2d+1;\\[0.18in]
\begin{array}{l}
O_a\left(\mathcal{S}_{1}(h) |I| T^{d+t-1}\right)
+O_a\left(\|h\|_\infty |I| T^{(1-\alpha)d+t}\right)
\end{array}
&\text{if }p+q=2d;\\[0.18in]
\begin{array}{l}
O_a\left(\mathcal S_1(h) |I| T^{t+d'-1}\log T\right)+
O_a\left(\mathcal S_1(h) |I| T^{d+t-2}\right)\\
\hspace{0.5in}+O_a\left(\|h\|_\infty |I| T^{(1-\alpha)(d-1)+t}\right),\\ 
\end{array}
&\text{if }p+q=2d-1;\\[0.18in]
\begin{array}{l}
O_a\left(\mathcal S_1(h)|I| T^{(1-\alpha)(p+q-2d)+t+d'-\alpha}\right)
\\
+O_a\left(\mathcal S_1(h) |I| T^{p+q+t-d-1}\right)+O_a\left(\|h\|_\infty |I| T^{(1-\alpha)(p+q-d)+t}\right),\\ 
\end{array}
&\\[0.15in]
&\hspace{-0.6in}\textcolor{black}{\text{if }
d+1 \le p+q < 2d-1,}
\\
\end{array}\right.
\end{split}\end{equation}
where
\[
J(h)=\frac 1 d \int_{\RR^t}\int_0^\infty \int_{S^{p-1}\times S^{q-1}}
h((\omega_1+\omega_2)r+Y'_1)r^{p+q-d-1} d\omega_1 d\omega_2 dr dY_1'
\]
and $\textcolor{black}{\mathcal S_1(h)}=\max(\|h\|_{\infty}, \|\partial h/\partial v_i\|_{\infty}, i=1, \ldots, n)$.
Here, $\omega_1$ and $\omega_2$ are spherical coordinates of unit spheres $S^{p-1}\subseteq \RR^p$ and $S^{q-1}\subseteq \RR^q$ with respect to the $\mathcal L^d$-norm, respectively.
\end{lemma}
\begin{proof} 
Let $(r_1, \omega_1)$ and $(r_2, \omega_2)$ be spherical coordinates of $\RR^p$ and $\RR^q$ with respect to the $\mathcal L^d$-norm, respectively so that
\[\begin{split}
&\mathcal I_{h,I}
=\int_{\RR^s}\int_{\RR^t}\int_{\RR^t}\int_0^\infty \int_0^\infty \int_{S^{p-1}\times S^{q-1}}
h\Big(\frac {r_1\omega_1+r_2\omega_2+Y_1+Y_2+Z}{T}\Big)\times\\
&\rchi_{I}(r_1^d-r_2^d+P_1(Y_1,Y_2)+P_2(Z),
Y_2,Z)\:
r_1^{p-1} r_2^{q-1}
d\omega_1 d\omega_2 dr_1 dr_2 dY_1 dY_2 dZ.
\end{split}\]

Let $T>1$ and let us divide the region $\{(r_1, r_2) \in \RR^2_{\ge0}\}$ into
\[
\{\: r_1 \ge r_2,\; 0\le r_2 \le T^{1-\alpha} \:\}\cup
\{\: r_2 \ge r_1,\; 0\le r_1 \le T^{1-\alpha}\:\}\cup
\{\: r_1, r_2 \ge T^{1-\alpha}\:\}.
\]

We will divide $\mathcal I_{h,I}$ into the summation of three integrals with respect to the partition above.

\vspace{0.1in}
\noindent (i) $r_1 \ge r_2$, $0 \le r_2 \le T^{1-\alpha}$.
Take
\begin{equation}\label{parametrization}
\left\{\begin{array}{ccl}
\textcolor{black}{\zeta}&=&r_1^d-r_2^d+P_1(Y_1,Y_2)+P_2(Z)\\
r&=&r_2/ T\\
Y'_1&=& Y_1/ T.
\end{array}\right.
\end{equation}

By change of variables,
\[\begin{split}
\Error{1}{h,I}&:=\int_{\RR^s}\int_{\RR^t}\int_{\RR^t}\int_{0}^{T^{1-\alpha}}\int_{r_{2}}^{\infty}\int_{S^{p-1}\times S^{q-1}}
h\Big(\frac {r_1\omega_1+r_2\omega_2+Y_1+Y_2+Z}{T}\Big)\times\\
&\rchi_{I}(r_1^d-r_2^d+P_1(Y_1,Y_2)+P_2(Z),
Y_2,Z)\:
r_1^{p-1} r_2^{q-1}
d\omega_1 d\omega_2 dr_1 dr_2 dY_1 dY_2 dZ\\
&=\frac {T^{p+q+t-d}} d\int_{\RR^s}\int_{\RR^t}\int_{\RR^t}\int_{-\infty}^\infty\int_{0}^{T^{-\alpha}}\int_{S^{p-1}\times S^{q-1}}
\\
&h\Big(\omega_1\Big(r^d+\frac {\zeta-P_1(TY'_1,Y_2)-P_2(Z)}{T^d}\Big)^{1/d}+\omega_2 r+Y'_1+\frac {Y_2+Z} T\Big)
\rchi_I (\zeta,Y_2,Z)\\
&\hspace{0.3in}\times \Big(r^d+\frac {\zeta-P_1(TY'_1,Y_2)-P_2(Z)}{T^d}\Big)^{(p-d)/d}r^{q-1}d\omega_1 d\omega_2 dr d\zeta dY'_1 dY_2 dZ.
\end{split}\]
\textcolor{black}{We remark that the domain of $(\zeta, r, Y'_1)$ in the second integral above is 
$$\left\{(\zeta, r, Y'_1) : \zeta + (Tr)^d- T(Y'_1 \cdot Y_2) - Q(Z) \ge T^{d(1-\alpha)}\right\}.$$ 
Since one can find $T_0>0$ such that this domain is contained in the intersection of supports of $h$ and $\rchi_{I}$ for $T>T_0$, we will refrain from referring to the above domain here and hereafter.}

\textcolor{black}{Recall that $(\zeta,Y_2,Z) \subseteq [-1,1]^{r+1}$ and $Y'_1\in B_a(0)$.} 
Also note that $\alpha<(d-d')/d$, $0 < r=\frac {r_2} T\le \frac 1 {T^{\alpha}}$ and  by the assumption on $P_1$, $0\le \frac {r_1^d-r_2^d} {T^{d'}}=\frac {\zeta-P_1(TY'_1,Y_2)-P_2(Z)} {T^{d'}}\ll a^{d'}$.
It follows that $r^d \le r^d+\frac {\zeta-P_1(TY'_1,Y_2)-P_2(Z)}{T^d} \ll \frac {a^{d'}} {T^{\alpha d}}$. Hence
\[
\Big(r^d+\frac {\zeta-P_1(TY'_1,Y_2)-P_2(Z)}{T^d}\Big)^{(p-d)/d} \ll
\left\{\begin{array}{cc}
\frac {a^{(p-d)\frac {d'} {d}}} {T^{\alpha(p-d)}}, &\text{if } p\ge d;\\[0.05in]
 r^{p-d}, &\text{if } p <d,
\end{array}\right.
\]
and we obtain that
\begin{equation}\begin{split}\label{eqn:smooth:1}
\Error{1}{h,I} &\ll
\|h\|_{\infty}|I|T^{p+q+t-d}\cdot a^{t}\times\left\{\begin{array}{cc}
\frac {a^{(p-d)\frac {d'} d}} {T^{\alpha(p-d)}}\int_0^{T^{-\alpha}} r^{q-1} dr, &\;\text{if }p \ge d;\\[0.05in]
\int_0^{T^{-\alpha}} r^{p+q-d-1} dr, &\;\text{if }p <d;
\end{array}\right.\\
&\ll_a \|h\|_{\infty}|I|T^{(1-\alpha)(p+q-d)+t}.
\end{split}\end{equation}

\noindent (ii) $r_1 \le r_2$, $0 \le r_1 \le T^{1-\alpha}$. \textcolor{black}{The second error is}
\[\begin{split}
\Error{2}{h,I}&:=\int_{\RR^s}\int_{\RR^t}\int_{\RR^t}\int_{0}^{T^{1-\alpha}}\int_{r_{1}}^{\infty}\int_{S^{p-1}\times S^{q-1}}
h\Big(\frac {r_1\omega_1+r_2\omega_2+Y_1+Y_2+Z}{T}\Big)\times\\
&\rchi_{I}(r_1^d-r_2^d+P_1(Y_1,Y_2)+P_2(Z),
Y_2,Z)\:
r_1^{p-1} r_2^{q-1}
d\omega_1 d\omega_2 dr_2 dr_1 dY_1 dY_2 dZ.\\
\end{split}\]

Similar to the case (i), by making the change of variables
\[
\left\{\begin{array}{ccl}
\zeta&=&r_1^d-r_2^d+P_1(Y_1,Y_2)+P_2(Z)\\
r&=&r_1/ T\\
Y'_1&=& Y_1/ T,
\end{array}\right.
\]
 it follows that
\begin{equation}\begin{split}\label{eqn:smooth:2}
\Error{2}{h,I} &\ll
\|h\|_{\infty}|I|T^{p+q+t-d}\cdot a^t\times\left\{\begin{array}{cc}
\frac {a^{(q-d)\frac {d'} d}} {T^{\alpha(q-d)}}\int_0^{T^{-\alpha}} r^{p-1} dr, &\;\text{if }q \ge d;\\[0.05in]
\int_0^{T^{-\alpha}} r^{p+q-d-1} dr, &\;\text{if }q <d;
\end{array}\right.\\
&\ll_a \|h\|_{\infty}|I|T^{(1-\alpha)(p+q-d)+t}.
\end{split}\end{equation}

Using the reparametrization in \eqref{parametrization} for the third range $\{r_1,r_2 \ge T^{1-\alpha}\}$, and by \eqref{eqn:smooth:1} and \eqref{eqn:smooth:2}, we have
\[\begin{split}
\mathcal I_{h,I}
&=\frac {T^{p+q+t-d}} d\int_{\RR^s}\int_{\RR^t}\int_{\RR^t}\int_{-\infty}^\infty \int_{1/T^{\alpha}}^\infty\int_{S^{p-1}\times S^{q-1}}\\
&h\Big(\omega_1\Big(r^d+\frac {\zeta-P_1(TY'_1,Y_2)-P_2(Z)}{T^d}\Big)^{1/d}+\omega_2 r+Y'_1+\frac {Y_2+Z} T\Big)
\rchi_I (\zeta,Y_2,Z)\\
&\hspace{0.6in}\times
\Big(r^d+\frac {\zeta-P_1(TY'_1,Y_2)-P_2(Z)}{T^d}\Big)^{(p-d)/d}r^{q-1}d\omega_1 d\omega_2 dr d\zeta dY'_1 dY_2 dZ\\
&+O_a\left(\|h\|_{\infty}|I|T^{(1-\alpha)(p+q-d)+t}\right).
\\
\end{split}\]

Now, note that since $r\ge \frac 1 {T^{\alpha}}$, we have $r^{d}T^{d-d'}\ge T^{(d-d')-d\alpha}$, where $(d-d')-d\alpha>0$ by the assumption on $\alpha$. 
Since $\frac {(\zeta-P_1(TY'_1,Y_2)-P_2(Z))} {T^{d'}}=O(a^{d'})$, 
there exists $T_{0}$ such that for $T>T_{0}$, $\left|\frac {\zeta-P_1(TY'_1,Y_2)-P_2(Z)} {T^{d'}}\right|<r^{d}T^{d-d'}$. Therefore for $T>T_{0}$,
\[
r\Big(1+ \frac{\zeta-P_1(TY'_1,Y_2)-P_2(Z)}{(rT)^d}\Big)^{1/d}
=r+O_a\Big(\frac 1 {r^{d-1}T^{d-d'}}\Big)
\]
and hence
\[\begin{split}
&h\Big(\omega_1\Big(r^d+\frac {\zeta-P_1(TY'_1,Y_2)-P_2(Z)}{T^d}\Big)^{1/d}+\omega_2 r+Y'_1+\frac {Y_2+Z} T\Big)\\
&\hspace{2in}=h((\omega_1+\omega_2)r+Y'_1)+O_a\Big(\frac{\mathcal S_1(h)}{r^{d-1}T^{d-d'}}\Big)+O\Big(\frac{\mathcal S_1(h)} {T}\Big).
\end{split}\]

Since $\supp h \subseteq B_a(0)$, and $\Big\|\omega_1\Big(r^d+\frac {\zeta-P_1(TY'_1,Y_2)-P_2(Z)}{T^d}\Big)^{1/d} +\omega_2 r+Y'_1+\frac {Y_2+Z} T\Big\|\ge r$, we obtain that
\[\begin{split}
\mathcal I_{h,I}
&=\frac {T^{p+q+t-d}} d\int_{\RR^s}\int_{\RR^t}\int_{\RR^t}\int_{-\infty}^\infty \int_{1/T^{\alpha}}^\infty\int_{S^{p-1}\times S^{q-1}}\\
&h\Big((\omega_1+\omega_2)r+Y'_1\Big)
\rchi_I (\zeta,Y_2,Z)\\
&\hspace{0.4in}\times
\Big(1+\frac {\zeta-P_1(TY'_1,Y_2)-P_2(Z)}{(rT)^d}\Big)^{(p-d)/d}r^{p+q-d-1}d\omega_1 d\omega_2 dr d\zeta dY'_1 dY_2 dZ\\
&\hspace{-0.3in}+O_a\Big(\mathcal S_1(h)|I| T^{p+q+t-2d+d'}\int_{1/T^{\alpha}}^a r^{p+q-2d} dr\Big)
+O_a\Big(\mathcal S_1(h)|I| T^{p+q+t-d-1}\int_{1/T^\alpha}^a r^{p+q-d-1}dr\Big) \\
&+O_a\left(\|h\|_{\infty}|I|T^{(1-\alpha)(p+q-d)+t}\right).
\end{split}\]

Similar to the above argument, since $\Big(1+ \frac{\zeta-P_1(TY'_1,Y_2)-P_2(Z)}{(rT)^d}\Big)^{(p-d)/d}=1+O\Big(\frac {a^{d'}} {r^dT^{d-d'}}\Big)$ for $T>T_{0}$ and $\supp h \subseteq B_a(0)$, we have
\[\begin{split}
\mathcal I_{h,I}
&=\frac {T^{p+q+t-d}} d\int_{\RR^s}\int_{\RR^t}\int_{\RR^t}\int_{-\infty}^\infty \int_{1/T^{\alpha}}^\infty\int_{S^{p-1}\times S^{q-1}}\\
&\hspace{1.1in}h\Big((\omega_1+\omega_2)r+Y'_1\Big)
\rchi_I (\zeta,Y_2,Z) r^{p+q-d-1}d\omega_1 d\omega_2 dr d\zeta dY'_1 dY_2 dZ\\
&+O_a\Big(\|h\|_{\infty} |I| T^{p+q+t-2d+d'}\int_{1/T^{\alpha}}^a r^{p+q-2d-1} dr\Big)\\
&\hspace{-0.3in}+O_a\Big(\mathcal S_1(h)|I| T^{p+q+t-2d+d'}\int_{1/T^{\alpha}}^a r^{p+q-2d} dr\Big)
+O_a\Big(\mathcal S_1(h)|I| T^{p+q+t-d-1}\int_{1/T^\alpha}^a r^{p+q-d-1}dr\Big) \\
&+O_a\left(\|h\|_{\infty}|I|T^{(1-\alpha)(p+q-d)+t}\right).
\end{split}\]

Finally, since $p+q+t=n-r$ and by definition of $J(h)$,
\[\begin{split}
&\frac {T^{p+q+t-d}} d\int_{\RR^s}\int_{\RR^t}\int_{\RR^t}\int_{-\infty}^\infty \int_{1/T^{\alpha}}^\infty\int_{S^{p-1}\times S^{q-1}}\\
&\hspace{1.1in}h\Big((\omega_1+\omega_2)r+Y'_1\Big)
\rchi_I (\zeta,Y_2,Z) r^{p+q-d-1}d\omega_1 d\omega_2 dr d\zeta dY'_1 dY_2 dZ\\
&=J(h)|I|T^{n-r-d}+O_a\left(\|h\|_{\infty}|I|T^{(1-\alpha)(p+q-d)+t}\right).
\end{split}\]

\textcolor{black}{One can obtain the equation \eqref{eq:smooth:3} after simplifying the error bounds in each case as follows:
\noindent (1) If $p+q\ge 2d+1$, using the fact that $d-d'\ge 1$ and $\mathcal S_1(h)\ge \|h\|_\infty$, we have
\[
\mathcal I_{h,I}=J(h)|I|T^{n-r-d}+ O_a\left(\mathcal S_1(h)|I|T^{p+q+t-d-1}\right)+O_a\left(\|h\|_\infty|I|T^{(1-\alpha)(p+q-d)+t}\right).
\]
\noindent (2) If $p+q=2d$, using the estimates $d'<(1-\alpha)d$, $d-d'\ge 1$ and $\mathcal S_1(h)\ge \|h\|_\infty$, we have
\[
\mathcal I_{h,I}=J(h)|I|T^{n-r-d}+O_a\left(\mathcal{S}_{1}(h) |I| T^{d+t-1}\right)
+O_a\left(\|h\|_\infty |I| T^{(1-\alpha)d+t}\right)
\]
\noindent (3) If $p+q=2d-1$, using $\alpha <1$ and $\mathcal S_1(h)\ge \|h\|_\infty$, we have
\[
\begin{split}
\mathcal I_{h,I}=J(h)|I|T^{n-r-d}
&+O_a\left(\mathcal S_1(h) |I| T^{t+d'-1}\log T\right)+
O_a\left(\mathcal S_1(h) |I| T^{d+t-2}\right)\\
&+O_a\left(\|h\|_\infty |I| T^{(1-\alpha)(d-1)+t}\right),\\ 
\end{split}
\]
\noindent (4) If $d+1\le p+q< 2d-1$, using $d'-d+\alpha d<0$ and $\mathcal S_1(h)\ge \|h\|_\infty$, we have
\[
\begin{split}
\mathcal I_{h,I}
=J(h)|I|T^{n-r-d}&+O_a\left(\mathcal S_1(h)|I| T^{(1-\alpha)(p+q-2d)+t+d'-\alpha}\right)
+O_a\left(\mathcal S_1(h) |I| T^{p+q+t-d-1}\right)\\
&+O_a\left(\|h\|_\infty |I| T^{(1-\alpha)(p+q-d)+t}\right),\\ 
\end{split}
\]
}
\end{proof}

\begin{theorem}\label{volume form} 
Let $p\ge 1,\; q\ge 1 \text{ with }d+1\le p+q\le n-r \text{ and let }$ $(F,M)\in \mathcal Y^{(F_0,M_0)}$.
Fix $N\ge 1$ and let $I\subseteq [-N,N]^{r+1}\subseteq \RR^{r+1}$ be measurable.
Let $0<\xi < \textcolor{black}{\min\left\{\frac {(d-d')(p+q-d)}{p+q-1}, \frac 1 2\right\}}$.
Then there exist $c^{}_{F,M}>0$ and $T_{0}>0$ such that for $T>T_{0}N$,
\[\begin{split}
&\vol\left((F,M)^{-1}(I)\cap B_{T}(0)\right)\\
&=c_{F,M}|I| T^{n-r-d}
+\left\{\begin{array}{lc}
O_{F,M}(|I|N^{1/2}T^{n-r-d-1/2}),
&\begin{array}{c}
\text{if } \textcolor{black}{\left[p+q > 2d-1\right]} \text{ or}\\
\left[p+q=2d-1 \text{ and } d-d'>1\right];
\end{array}\\[0.18in]
O_{F,M}(|I|N^{1/2}T^{n-r-d-1/2}\log T),
&\begin{array}{c}
\text{if } p+q=2d-1 \text{ and}\\
d-d'=1;
\end{array}\\[0.18in]
O_{F,M}(|I|N^{\xi}T^{n-r-d-\xi}),
&\text{if } d+1 \le p+q < 2d-1.\\
\end{array}\right.\end{split}\]
\end{theorem}

\begin{proof}
Let $(F,M)=(\lambda F_{0}^{g_{1}},g_1\M_{0}g_2)$, where $(\lambda, g_{1}, g_2)\in (\RR-\{0\})\times \SL_n(\RR)\times \GL_{r}(\RR)$. 
By replacing $I$ by ${\tiny \left(\hspace{-0.04in}\begin{array}{cc} \lambda & 0 \\ 0 & g_2\end{array}\hspace{-0.04in}\right)^{-1}}\hspace{-0.01in}I$ \textcolor{black}{and $N$ by $\max(\lambda^{-1},\|g_2^{-1}\|)N$,} we may assume that $(\lambda, g_2)=(1, id)$.
We first assume that $N=1$. 

Take $h=\rchi_{B_1(0)}$ the indicator function of the unit ball in $\RR^{n}$.
For $\delta\in (0,1)$, let $h^{\pm}_{\delta}$ be smooth functions on $\RR^{n}$ such that $h^{\pm}_{\delta} \in (0,1)$ and 
\[
h^+_{\delta}(v)=
\left\{\begin{array}{cl}
1, & \text{if } \|v\|\le 1;\\
0, & \text{if } \|v\|\ge 1+\delta, \end{array}\right.
\quad\text{and}\quad
h^-_{\delta}(v)=
\left\{\begin{array}{cl}
1, & \text{if } \|v\|\le 1-\delta;\\
0, & \text{if } \|v\|\ge 1. \end{array}\right.
\] 
We may further assume that $\mathcal S_1(h^{\pm}_{\delta}) \ll 1/\delta$.

Define $h^{\pm}_{\delta,g_1}(v)=h^{\pm}_{\delta}(v g_1^{-1})$. Note that $ \text{supp}(h_{\delta, g_{1}}^{\pm})\subseteq B_{2\|g_1\|}$ and $\mathcal{S}_{1}(h_{\delta,g_{1}}^{\pm})=O(\|g_{1}\| \delta^{-1})$.
Applying Lemma $\ref{smooth version}$ to functions $h_{\delta,g_{1}}^{\pm}$, there is $T_{0}>0$ such that for $T>T_{0}$,
\begin{equation}\label{eq:volume:1}
\begin{split}
&\int_{\RR^n} h_{\delta,g_{1}}^{\pm}\Big(\frac v T\Big)\rchi_{I}(Q_0(v),\M_0(v)) dv
=J(h_{\delta,g_{1}}^{\pm})|I|\:T^{n-r-d}\\
&+\left\{\begin{array}{lc}
\begin{array}{l}
O_{g_1}\left(\delta^{-1}|I|T^{p+q+t-d-1}\right)\\
\hspace{0.5in}+O_{g_1}\left(|I|T^{(1-\alpha)(p+q-d)+t}\right),
\end{array}
&\text{if }p+q \ge 2d+1;\\[0.18in]
\begin{array}{l}
O_{g_1}\left(\delta^{-1} |I| T^{d+t-1}\right)
+O_{g_1}\left(|I| T^{(1-\alpha)d+t}\right)
\end{array}
&\text{if }p+q=2d;\\[0.18in]
\begin{array}{l}
O_{g_1}\left(\delta^{-1} |I| T^{t+d'-1}\log T\right)+
O_{g_1}\left(\delta^{-1} |I| T^{d+t-2}\right)\\
\hspace{0.5in}+O_{g_1}\left(|I| T^{(1-\alpha)(d-1)+t}\right),\\ 
\end{array}
&\text{if }p+q=2d-1;\\[0.18in]
\begin{array}{l}
O_{g_1}\left(\delta^{-1}|I| T^{(1-\alpha)(p+q-2d)+t+d'-\alpha}\right)
\\
+O_{g_1}\left(\delta^{-1}|I| T^{p+q+t-d-1}\right)+O_{g_1}\left(|I| T^{(1-\alpha)(p+q-d)+t}\right),\\ 
\end{array}
&\\[0.15in]
&\hspace{-0.6in}\textcolor{black}{\text{if }
d+1 \le p+q < 2d-1.}
\\
\end{array}\right.
\end{split}
\end{equation}

It is obvious that 
\begin{equation}\label{eq:volume:3}
\mathcal I_{h^-_{\delta, g_1},\; I}
\le \vol((F,\M)^{-1}(I)\cap B_T(0)) 
\le \mathcal I_{h^+_{\delta, g_1},\; I}
\end{equation}
and
\[J(h^+_{\delta, g_1}) \le J(h_{g_1}) \le J(h^-_{\delta, g_1}),
\]
where $h_{g_1}(v)=h(v g_1^{-1})$, and $\mathcal I_{h^{\pm}_{\delta, g_1},\; I}$ and $J(h^{\pm}_{\delta, g_1})$ are defined as in Lemma \ref{smooth version}.

We claim that 
\begin{equation}\label{eq:volume:2}
J(h_{g_1})-J(h^{\pm}_{\delta,g_1})=O_{g_1}(\delta)
\end{equation}
as $\delta$ goes to zero.
Denote by $(r_3, \omega_3)$ the spherical coordinates of $\RR^{t}$ with respect to $\mathcal L^d$-norm, so that $dY'_1=r_3^{t-1}dr_3 d\omega_3$.
Define the new coordinates $(R, \omega)$ of $\left((S^{p-1}\times S^{q-1})\times \textcolor{black}{\mathbb R_{>0}}\right)\times \RR^t$ by
\[\begin{split}
R&=\textcolor{black}{\sqrt[d]{2r^d+r_3^d}};\\
\omega&=(\omega_1, \omega_2, \omega_3, \theta_1,\theta_2),
\end{split}\]
where \textcolor{black}{$(R,\theta=(\theta_1,\theta_2))$ is the spherical coordinates of $\{(\sqrt[d]{2} r, r_3)\in \RR^2_{>0}\}$} with respect to $\mathcal L^d$-norm. Note that $\omega$ is the coordinates of the compact set 
\[
\SS:=\left\{(r, r_3, \omega_1, \omega_2, \omega_3)\in \left((S^{p-1}\times S^{q-1})\times \textcolor{black}{\RR_{>0}}\right)\times \RR^t : 2r^d + r_3^d=1 \;\right\}.
\]

For each $\omega \in \SS$, $J(h_{g_1})-J(h^{\pm}_{\delta,g_1})\neq 0$ only if
$
\|\left((\omega_1+\omega_2)r+Y'_1\right)g_1^{-1}\|
=\|R\omega\: g_1^{-1}\|\in (1-\delta, 1+\delta),
$
or equivalently, 
\[
J(h_{g_1})-J(h^{\pm}_{\delta, g_1})\neq 0\;\text{only if}\;
\textcolor{black}{(1-\delta)\|\omega g_1^{-1}\|^{-1}
\le R
\le (1+\delta)\|\omega g_1^{-1}\|^{-1}.}
\]

Since $\textcolor{black}{drdr_3\ll R\:dRd\theta}$, we have
\begin{equation*}\begin{split}
|J(h_{g_1})-J(h^{\pm}_{\delta,g_1})| 
&\ll \int_{\SS}\int_{(1-\delta)\textcolor{black}{\|\omega g_1^{-1}\|^{-1}}}^{(1+\delta)\textcolor{black}{\|\omega g_1^{-1}\|^{-1}}}
 R^{p+q+t-d-1} \textcolor{black}{\theta_1}^{p+q-d-1} \textcolor{black}{\theta_2}^{t-1} dRd\omega \\
&\ll \int_{\SS}\int_{(1-\delta)\textcolor{black}{\|\omega g_1^{-1}\|^{-1}}}^{(1+\delta)\textcolor{black}{\|\omega g_1^{-1}\|^{-1}}}
R^{p+q+t-d-1} dRd\omega\\
&\ll \int_{\SS} \textcolor{black}{\|\omega g_1^{-1}\|^{-(p+q+t-d)}} \delta d\omega \ll \|g_1^{-1}\|^{-(p+q+t-d)} \delta.
\end{split}\end{equation*}
Here we use the fact that $p+q> d$. We consider several cases.

\vspace{0.2in}
\noindent i) $p+q\ge 2d-1\;(p+q\neq 3)$.

Take $\delta=T^{-1/2}$ and choose any $\alpha$ in $\left[\frac 1 {2(p+q-d)}, \frac {d-d'} d\right)$.
Then by \eqref{eq:volume:1} and \eqref{eq:volume:2},
\[
\mathcal I_{h^\pm_{\delta,g_1},\;I}
=\left\{\begin{array}{lc}
J(h_{g_1})|I|T^{n-r-d}+O_{g_1}\left(|I|T^{n-r-d-1/2}\log T\right),
&\begin{array}{c}
\text{if }p+q=2d-1 \text{ and }\\
 d'=d-1;
 \end{array} \\[0.18in]
J(h_{g_1})|I|T^{n-r-d}+O_{g_1}\left(|I|T^{n-r-d-1/2}\right),
&\text{otherwise.}
\end{array}\right.
\]

Hence the result follows from \eqref{eq:volume:3} and by putting $c^{}_{F,M}=J(h_{g_1})$.

\vspace{0.2in}
\noindent ii) $d+1\le p+q < 2d-1$.

For $0<\xi<\textcolor{black}{\min\left\{\frac {(d-d')(p+q-d)}{p+q-1}, \frac 1 2\right\}}$,
take $\delta=T^{-\xi}$ and pick any $\alpha$ in $\Big[\frac{\xi}{p+q-d},\textcolor{black}{\min\left\{\frac{d-d'}{d},\frac{d-d'-2\xi}{2d-1-p-q}\right\}}\Big)$. 
Again, by \eqref{eq:volume:1}, \eqref{eq:volume:2} and \eqref{eq:volume:3}, we have
\[
\vol\left((F,M)^{-1}(I)\cap B_{T}(0)\right)= c_{F,M}|I|T^{n-r-d}+O_{F,M}\left(|I|T^{n-r-d-\xi}\right),
\]
where $c_{F,M}=J(h_{g_1})$.

\vspace{0.2in}
Now consider the case when $I\subseteq [-N,N]^{r+1}$ for $N\geq 1$. Define
\[
I'=\left\{\left(\frac{a_{0}}{N^{d}},\frac{a_{1}}{N},...,\frac{a_{r}}{N}\right) :  (a_{0},a_{1},...,a_{r})\in I \right\}.
\]

Then $I'\subseteq [-1,1]^{r+1}$, $|I'|=N^{-(r+d)}|I|$ and $(F,M)(v)\in I$ if and only if $(F,M)(N^{-1}v)\in I'$. By taking $w=N^{-1}v$, it follows that
\[
(F,M)^{-1}(I)\cap B_{T}(0)=\left\{Nw\;|\; (F,M)(w)\in I', \|w\|\leq N^{-1}T \right\}=N\left((F,M)^{-1}(I')\cap B_{N^{-1}T}\right).
\]

Therefore, the theorem is deduced from the case of $N=1$.
\end{proof}

\begin{corollary}\label{Simple}
	Let $p\ge1,\; q\geq 1$ with $d+1\le p+q\le n-r$ and let $(F,M) \in \mathcal Y^{(F_0,M_0)}$.
	Fix $N\ge 1$ and let $I\subseteq [-N,N]^{r+1}\subseteq \RR^{r+1}$ be measurable.
	Let $0<\xi < \frac{1}{2d-2}$.
	Then there exist $c^{}_{F,M}>0$ and $T_{0}>0$ such that for $T>T_{0}N$,
\[
\vol\left((F,M)^{-1}(I)\cap B_{T}(0)\right)=c_{F,M}|I|T^{n-r-d}+O_{F,M}\left(|I|N^{\xi}T^{n-r-d-\xi}\log T\right),
\]
where $\log T$ occurs only when $p+q=2d-1$ and $d-d'=1$.
\end{corollary}
\begin{proof}
	It follows easily from Theorem \ref{volume form}. Note that for $d+1\le p+q<2d-1$, $\frac{1}{2d-2}<\frac{(d-d')(p+q-d)}{p+q-1}$ and for $T>T_{0}N$, $N^{1/2}T^{n-r-d-1/2}\ll N^{\xi}T^{n-r-d-\xi}$.
\end{proof}

In particular, taking $\xi=\frac{1}{2d}$ in the above corollary, we get that

\begin{equation}\label{eq:simple}
	\vol\left((F,M)^{-1}(I)\cap B_{T}(0)\right)=c_{F,M}|I|T^{n-r-d}+O_{F,M}\left(|I|N^{1/2d}T^{n-r-d-1/2d}\log T\right).
\end{equation}

We will use this simplified version of volume form in subsequent theorems.

\section{Discrepancy estimates}\label{Section 3}
\textcolor{black}{In this section, we prove Theorem \ref{thm:main1_new} and Corollary \ref{cor:main1_new}. We continue to use the notations in Section \ref{sec:classification}.
We will regard $\mathcal Y^{(F_0,M_0)}$ as the union of $\SL_n(\RR)$-slices}
\[\mathcal Y_{g_2^{\lambda}}:=\left\{(\lambda F_0^{g_1},g_1\M_0 g_2) : g_1\in \SL_n(\RR)\right\},\]
\textcolor{black}{over all $(\lambda, g_2) \in (\RR-\{0\})\times \GL_r(\RR)$. 
Since an $(\RR-\{0\})\times \SL_n(\RR)\times \GL_r(\RR)$-invariant measure on $\mathcal Y^{(F_0,M_0)}$ is the push-forward measure of the product of Haar measures on $\RR-\{0\}$, $\SL_n(\RR)$ and $\GL_r(\RR)$, respectively, 
it suffices to show that \textcolor{black}{for all $(\lambda, g_2)\in (\RR-\{0\})\times \GL_r(\RR)$} and for almost every $(F,M)\in \mathcal Y_{g_2^{\lambda}}$, the statements of Theorem \ref{thm:main1_new} and Corollary \ref{cor:main1_new} holds.}

\begin{theorem}\label{thm:main1}
		Let $0\le \kappa<n-r-d$. Let $\{I_{t}\}_{t>0}$ be a \textcolor{black}{non-increasing} family of bounded measurable subsets of $\RR^{r+1}$ with $|I_{t}|=\textcolor{black}{ct^{-\kappa}}$ for some $c>0$. Then there is $\nu >0$ such that for almost every $(F,M)\in \mathcal{Y}_{g_2^{\lambda}}$, there exists $c_{F,M}>0$ such that 
		\[
		\#\{v\in \ZZ^{n} : (F,M)(v)\in I_{t}, \|v\|\le t\}=c_{F,M}|I_{t}|t^{n-r-d}+O_{F,M}(t^{n-r-d-\kappa-\nu}).
		\]
	\end{theorem}

\vspace{0.2in}
Given a lattice $\Lambda \subset \RR^{n}$ and a finite volume set $A\subseteq \RR^{n}$, we define \textcolor{black}{\emph{the discrepancy of lattice points in $A$}} by
\begin{equation}\label{def:disc}
D(\Lambda, A)=|\#(\Lambda\cap A)-\vol(A)|.
\end{equation}

\begin{lemma}\label{Discrepancy}
	Let $0\le \kappa<n-r-d$. Let $\{I_{t}\}_{t>0}$ be a \textcolor{black}{non-increasing} family of bounded measurable subsets of $\RR^{r+1}$ with $|I_{t}|=\textcolor{black}{ct^{-\kappa}}$ for some $c>0$. Then there exists some $\delta \in (0,1)$ such that for almost every $g\in \SL_n(\RR)$ there exists $t_{g}>0$ such that for all $t\geq t_{g}$,
\begin{equation}	
D(\ZZ^{n}g,A_{g,I_{t},t})< \vol(A_{g,I_{t},t})^{\delta},
\end{equation}	
where $A_{g,I_{t},t}:=(F_{0},M_{0})^{-1}(I_{t})\cap B_{t}(0)g$ with $(F_0,M_0)$ as in \eqref{eqn:def}.
	\end{lemma}
	\begin{proof}
		Let $\mathcal{K}\subset \SL_n(\RR)$ be compact and consider a sequence $\{t_{k}=k^{\alpha}\}_{k\in \NN}$, where $\alpha >\max\left\{2d,\frac{d_{n}+3}{n-r-\kappa-d}\right\}$ with $d_{n}=\frac 1 2 (n+2)(n-1)$.
		Let $\delta_{0}=1-\frac{1}{\alpha(n-r-\kappa-d)}$. For $t>0$ and $\delta \in (\delta_{0},1)$, define 
		\[\mathcal{B}_{t}:=\{g\in \mathcal{K} : D(\ZZ^{n}g,A_{g,I_{t},t})\geq \vol(A_{g,I_{t},t})^{\delta} \}.\]	
		Then it suffices to show that $\overline{\lim}_{t\rightarrow \infty} \mathcal B_t$ is a null set.

	   For each $k\in \NN$, let $\epsilon_k=1/k$. 
	   By Lemma 2.1 in \cite{KY18}, there is a finite subset $\mathcal I_k \subset \mathcal K$ with $\# \mathcal I_k=O_{\mathcal K}(k^{d_n})$, $d_n=(n+2)(n-1)/2$ such that $\mathcal K \subset \bigcup_{h\in \mathcal I_k} \mathcal U_{\epsilon_k}h$.
	   Here, $\mathcal U_{\epsilon_k}$ is a $\epsilon_k$-neighborhood of the identity in $\SL_n(\RR)$, with respect to the operator norm (acting on $\RR^n$). For each $k\in \NN$ and $h\in \mathcal I_k$, set
	   $$\underline{A}_{k,h}=(F_{0},M_{0})^{-1}(I_{t_{k+1}})\cap B_{(1-\epsilon_{k})t_{k}}h,$$ 		 $$\overline{A}_{k,h}=(F_{0},M_{0})^{-1}(I_{t_{k}})\cap B_{(1+\epsilon_{k})t_{k+1}}h.$$
		 Then for $t_{k}\le t<t_{k+1}$ and $g\in \mathcal U_{\epsilon_k}h$, we get $\underline A_{k,h} \subseteq A_{g,I_t,t} \subseteq \overline A_{h,k}$. Let
		 $$T_{k,h}=\vol(\underline{A}_{k,h})^{\delta}-\vol(\overline{A}_{k,h}\setminus\underline{A}_{k,h}).$$ 
		 
		 Define
   	 $$\mathcal{M}_{A,T}:=\{g\in \mathcal{K} : D(\ZZ^{n}g,A)\geq T \}.$$
   	 Then Theorem 2.2 of \cite{KY18}, which is deduced from Rogers' second moment formula \cite{Rog55}, says that 
   	 \begin{equation}\label{eqn:secmom}
   	 \mu(\mathcal M_{A,T})\ll_{\mathcal K} \frac{\vol(A)} {T^2},
   	 \end{equation}
   	 \textcolor{black}{where $\mu$ is the normalized $\SL_n(\RR)$-invariant measure on $\SL_n(\ZZ)\setminus\SL_n(\RR)$.}
   	 
	    Moreover, as in the proof of Theorem 6 of \cite{KY18}, we have that
	   \[\bigcup\limits_{t_{k}\leq t <t_{k+1}}\mathcal{B}_{t}\subset \bigcup\limits_{h\in \mathcal{I}_{k}}\mathcal{C}_{k,h},\]
		where $\mathcal{C}_{k,h}=\mathcal{M}_{\underline{A}_{k,h},T_{k,h}}\cup \mathcal{M}_{\overline{A}_{k,h},T_{k,h}}.$
		
		By equation \eqref{eq:simple}, since $|I_{t_{k}}|=\textcolor{black}{ct_{k}^{-\kappa}}$, $(1+\epsilon_{k})^{n-r-d}=\left(1+\frac {1}{k}\right)^{n-r-d}=1+O\left(\frac {1}{k}\right)$ 
		and $t_{k+1}^{n-r-d}=(k+1)^{\alpha(n-r-d)}=k^{\alpha(n-r-d)}\left(1+O_{\kappa}\left(\frac {1}{k}\right)\right),$ we get that
		\[
		\vol(\overline{A}_{k,h})=c_{h}c\:k^{\alpha(n-r-\kappa-d)}\\
		+O_{\mathcal{K},c, \kappa}\left(k^{\alpha(n-r-\kappa-d)-1}+k^{\alpha(n-r-\kappa-d-1/2d)}\log(k+1)\right).
		\]

	Since $\alpha>2d$, we have $\alpha(n-r-\kappa-d)-1>\alpha(n-r-\kappa-d-1/2d)$ and hence
		\[
		\vol(\overline{A}_{k,h})=c_{h}c\:k^{\alpha(n-r-\kappa-d)}+O_{\mathcal{K},c,\kappa}\left(k^{\alpha(n-r-\kappa-d)-1}\right).
		\]
		
		Similarly, for sufficiently large $k$,
		\[
		\vol(\underline{A}_{k,h})=c_{h}c\:k^{\alpha(n-r-\kappa-d)}+O_{\mathcal{K},c,\kappa}\left(k^{\alpha(n-r-\kappa-d)-1}\right).
		\]
		\textcolor{black}{Therefore,
		\[
		\vol(\underline{A}_{k,h})^{\delta}\asymp_{\mathcal{K},c} k^{\delta \alpha(n-r-\kappa-d)}\text{ and } \vol(\overline{A}_{k,h}\setminus\underline{A}_{k,h})\ll_{\mathcal{K},c,\kappa} k^{ \alpha(n-r-\kappa-d)-1}.
		\]
       Since $\delta>1-\frac{1}{\alpha(n-r-\kappa-d)}$, we get
		\[
		T_{k,h}=\vol(\underline{A}_{k,h})^{\delta}-\vol(\overline{A}_{k,h}\setminus\underline{A}_{k,h})\asymp_{\mathcal{K},c,\kappa} k^{\delta \alpha(n-r-\kappa-d)}.
		\]
	}
		By using \eqref{eqn:secmom}, we have
		\[
		\mu(\mathcal{C}_{k,h})\ll_{\mathcal{K}}\dfrac{\vol(\overline{A}_{k,h})}{T_{k,h}^{2}}\asymp_{\mathcal{K},c,\kappa} k^{\alpha(n-r-\kappa-d)(1-2\delta)}.
		\]
		Now since $\# \mathcal{I}_{k}=O(k^{d_{n}})$, we get
		\[
		\mu( \bigcup\limits_{t_{k}\leq t <t_{k+1}}\mathcal{B}_{t})\leq \sum\limits_{h\in \mathcal{I}_{k}}\mu(\mathcal{C}_{k,h})\ll_{\mathcal{K},c,\kappa}\frac{1}{k^{\alpha(n-r-\kappa-d)(2\delta-1)-d_{n}}}.		\]
		Since $\alpha>\max\left\{2d, \frac{d_n+3}{n-r-\kappa-d} \right\}$ \textcolor{black}{and $\delta >\delta_{0}$} we have $\alpha(n-r-\kappa-d)(2\delta-1)-d_{n}>1$ and hence
		\[\mu\left(\overline{\lim_{t\rightarrow \infty}}\mathcal{B}_{t}\right)
		\leq \lim_{m\rightarrow\infty}\sum_{k=m}^{\infty}\mu\Big( \bigcup\limits_{t_{k}
		\leq t <t_{k+1}}\mathcal{B}_{t}\Big)\leq \lim_{m\rightarrow \infty}\sum_{k=m}^{\infty}\frac{1}{k^{\alpha(n-r-\kappa-d)(2\delta-1)-d_{n}}}=0.\]
		
		Hence \textcolor{black}{for given $\delta\in (\delta_{0},1)$, it follows that} for almost every $g\in \SL_n(\RR)$ and for sufficiently large t, $	D(\ZZ^{n}g,A_{g,I_{t},t})< \vol(A_{g,I_{t},t})^{\delta}.$
	\end{proof}
	
We are now ready for the proof of Theorem \ref{thm:main1}.
\begin{proof}[Proof of Theorem \ref{thm:main1}] 
		For $\delta \in (\delta_{0},1)$, where $\delta_{0}$ is as in the proof of Lemma \ref{Discrepancy}, let $$0<\nu<(1-\delta)(n-r-\kappa-d).$$ 
		
		By replacing $\{I_{t}\}$ with $\{I_t(g_2^{\lambda})^{-1}\}$  in Lemma \ref{Discrepancy},
		we obtain that for almost every $g_{1}\in \SL_n(\RR)$ and for sufficiently large t,
		\begin{equation}\label{eqn:mainthm:1}
		D(\ZZ^{n}g_{1},A_{g_{1},I_t (g_2^{\lambda})^{-1},t})< \vol(A_{g_{1},I_t (g_2^{\lambda})^{-1},t})^{\delta}.\end{equation}
		
		For $(F,M)\in \mathcal{Y}_{g_2^{\lambda}}$, denote 
		$$\mathcal{N}_{(F,M)}(I_{t},t):= \#\{v\in \ZZ^{n} : (F,M)(v)\in I_{t}, \|v\|\le t\}.$$ 
		
		Then,
		\[
		\begin{split}
		\mathcal{N}_{(F,M)}(I_{t},t)&= \# (\ZZ^{n}\cap (F,M)^{-1}(I_{t})\cap B_{t}(0))\\
		&=\# (\ZZ^{n}g_{1}\cap (F_{0},M_{0})^{-1}(I_t (g_2^{\lambda})^{-1})\cap B_{t}(0)g_{1})\\
		&=\#(\ZZ^{n}g_{1}\cap A_{g_{1},I_t(g_2^{\lambda})^{-1},t}).
		\end{split}
		\]
		
		Hence for almost every $(F,M)\in \mathcal{Y}_{g_2^{\lambda}}$ and for sufficiently large t,
		\[
		\begin{split}
		&\left|\mathcal{N}_{(F,M)}(I_{t},t)-c_{F,M}|I_{t}|t^{n-r-d}\right|\\
		&\hspace{0.4in}\leq D(\ZZ^{n}g_{1},A_{g_{1},I_t (g_2^{\lambda})^{-1},t})+\left|\vol(A_{g_{1},I_t (g_2^{\lambda})^{-1},t})-c_{F,M}|I_{t}|t^{n-r-d}\right|\\
		&\hspace{0.4in}\leq \vol(A_{g_{1},I_t (g_2^{\lambda})^{-1},t})^{\delta}+
		O_{F,M,c}\left(t^{n-r-d-\kappa-1/2d}\log t\right)\\
		&\hspace{0.4in}\leq (2c_{F,M}ct^{n-r-d-\kappa})^{\delta}
		+O_{F,M,c}\left(t^{n-r-d-\kappa-1/2d}\log t\right)\\
		&\hspace{0.4in}\le t^{n-r-d-\kappa-\nu},
		\end{split}
		\]
		thus proving the theorem.
	\end{proof}
As an immediate corollary we  have,	
\begin{corollary}\label{cor:main1}
	Let $0\le\kappa<n-r-d$ and $(\kappa_{0},\kappa_{1},\ldots,\kappa_{r})\in \RR_{\ge 0}^{r+1}$ be such that $\sum_{i=0}^r \kappa_{i}=\kappa$. 
	Then for any $\xi=(\xi_0,\xi_1,\ldots,\xi_{r})\in \RR^{r+1}$ and for almost every $(F,M)\in \mathcal{Y}_{g_2^{\lambda}}$, the system of inequalities
	\[\left\{\begin{array}{l}
	|F(v)-\xi_0|<t^{-\kappa_0};\\
	|l_i(v)-\xi_{i}|<t^{-\kappa_{i}},\;1\le i\le r;\\
	\|v\|\le t,
	\end{array}\right.\]
	has integer solutions for sufficiently large t, \textcolor{black}{where $(l_1,\ldots,l_r)=M.$ }
\end{corollary}

\section{Uniform approximation}
As in Section \ref{Section 3}, to show Theorem \ref{thm:main2_new} and Corollary \ref{cor1.4}, it suffices to show the theorems \textcolor{black}{for all $(\lambda, g_2)\in (\RR-\{0\})\times \GL_r(\RR)$} and for almost all $(F,\M)\in \mathcal Y_{g_2^\lambda}$.
Let us first show the following theorem.

\begin{theorem} \label{uniform}
	 Let $0\leq \eta < \min\left\{1,\frac{n-r-d}{r+1}\right\}$ and $0\le\kappa< \frac{n-r-d-(r+1)\eta}{r+2}.$ 
	 Let $N(t)$ be a non-decreasing function such that $N(t)=O(t^{\eta})$. 
	 Then there exists some $\delta \in (0,1)$ such that 
	 for almost every $(F,M)\in \mathcal{Y}_{g_2^{\lambda}}$, there exists $t_{F,M}>0$ satisfying the following:
	 for all $t>t_{F,M}$ and for \textcolor{black}{all $I\subset [-N(t),N(t)]^{r+1}$} of the form $I=I_0\times I_1\times \cdots \times I_{r}$ with $|I_j|=t^{-\kappa_j}$ \textcolor{black}{where $\kappa_{j}\ge 0$} and $\sum_{j=0}^{r}\kappa_j =\kappa$,
	\begin{equation}\label{eqn:thm4.1:1}
	\left|\mathcal{N}_{F,M}(I,t)-c_{F,M}t^{n-r-d-\kappa}\right|<t^{\delta(n-r-d-\kappa)}.
	\end{equation}
	\end{theorem}
	\begin{proof}
		Let $\mathcal{K}\subset \SL_n(\RR)$ be compact and $\{t_k=k^{\alpha}\}_{k\in \NN}$ be a sequence with $$\alpha>\max\left\{\frac{2d}{1-\eta}, \frac{d_n+3}{n-r-d-(r+1)\eta-(r+2)\kappa}\right\},$$ where $d_n=\frac{(n+2)(n-1)}{2}$ and let $\delta_0=1-\frac{1}{\alpha(n-r-d-\kappa)}.$ For $t>0$ and $\delta\in (\delta_0,1)$, define
		
			\[
		\mathcal{B}_{t}:=\left\{g\in \mathcal{K}:\;
		 \begin{split}&\exists~I=I_0\times I_1\times \cdots \times I_{r}\subset [-N(t),N(t)]^{r+1}\text{ with } |I_j|=t^{-\kappa_j} \\
		& \text{such that }D(\ZZ^{n}g,A_{g,I,t})\geq \vol(A_{g,I,t})^{\delta} \end{split}\right\},
		\]
		where \textcolor{black}{$D(\ZZ^{n}g,A_{g,I,t})$ is as in \eqref{def:disc}} and in this time, we let $A_{g,I,t}=(F_{0},M_0)^{-1}(I{(g_2^\lambda})^{-1})\cap B_t(0) g$ so that the set of $g\in \mathcal K$ for which \eqref{eqn:thm4.1:1} does not hold is contained in $\overline{\lim}_{t\rightarrow \infty} \mathcal B_t$.
		
		Let $\beta=\max\{\kappa_j\}_{j=0}^r$. 
		For any $t_k \le t < t_{k+1}$, take a $t_{k+1}^{-\beta}$-dense partition of the interval $[-N(t_{k+1}),N(t_{k+1})]$:
				$$-N(t_{k+1})=\xi_{k,0}<\xi_{k,1}<\ldots <\xi_{k,M(k)}=N(t_{k+1}).$$

For any \textcolor{black}{subset} $I=I_0\times I_1\times \cdots \times I_{r}\subset [-N(t),N(t)]^{r+1}\text{ with } |I_j|=t^{-\kappa_j}$ for $0\leq j\leq r$, its center point $\xi=(\xi_0,\ldots,\xi_{r})$ lies in $[-N(t_{k+1}),N(t_{k+1})]^{r+1}$.  
Therefore, there exists $(i_0,\ldots,i_{r})$ with $0\le i_j < M(k)$ such that $\xi_{k,i_{j}}\le \xi_j< \xi_{k,i_{j}+1}$ for $0\le j\le r$. 

Since 
$I_j=(\xi_j-\frac{t^{-\kappa_j}}{2},\xi_j+\frac{t^{-\kappa_j}}{2})$ and 
$\xi_{k,i_j+1}-\xi_{k,i_j}<t_{k+1}^{-\beta}\le t_{k+1}^{-\kappa_j}$, it follows that
		\[
		\left(\xi_{k,i_j+1}-\frac{t_{k+1}^{-\kappa_j}}{2},\xi_{k,i_j}+\frac{t_{k+1}^{-\kappa_j}}{2}\right)\subseteq I_j \subseteq \left(\xi_{k,i_j}-\frac{t_{k}^{-\kappa_j}}{2},\xi_{k,i_j+1}+\frac{t_{k}^{-\kappa_j}}{2}\right).
		\]
		
		Let $I_j^{1}=	\left(\xi_{k,i_j+1}-\frac{t_{k+1}^{-\kappa_j}}{2},\xi_{k,i_j}+\frac{t_{k+1}^{-\kappa_j}}{2}\right)$ and $I_j^{2}=\left(\xi_{k,i_j}-\frac{t_{k}^{-\kappa_j}}{2},\xi_{k,i_j+1}+\frac{t_{k}^{-\kappa_j}}{2}\right)$,
		and denote $I^{1}_{(\vi)}=\prod_{j=0}^{r}I_j^{1}$ and $I^{2}_{(\vi)}=\prod_{j=0}^{r}I_j^{2}$ so that $I^{1}_{(\vi)}\subseteq I \subseteq I^{2}_{(\vi)}$.
		
		Similar to the proof of Lemma \ref{Discrepancy}, we have
			\[
		\bigcup\limits_{t_{k}\leq t <t_{k+1}}\mathcal{B}_{t}\subseteq \bigcup\limits_{h\in \mathcal{I}_{k}}\bigcup\limits_{i_0=0}^{M(k)-1}\ldots \bigcup\limits_{i_{r}=0}^{M(k)-1}\mathcal{C}_{k,\vi,h},
		\]
		where $\mathcal{C}_{k,\vi,h}=\mathcal{M}_{\underline{A}_{k,\vi,h},T_{k,\vi,h}}\cup \mathcal{M}_{\overline{A}_{k,\vi,h},T_{k,\vi,h}}$
		with 
		$$\underline{A}_{k,\vi,h}=(F_0,M_0)^{-1}\left(I^{1}_{(\vi)}(g_2^\lambda)^{-1} \right)\cap B_{(1-\epsilon_k)t_k}h,$$
		$$\overline{A}_{k,\vi,h}=(F_0,M_0)^{-1}\left(I^{2}_{(\vi)}(g_2^\lambda)^{-1}\right)\cap B_{(1+\epsilon_k)t_{k+1}}h,$$
		$$T_{k,\vi,h}=\vol(\underline{A}_{k,\vi,h})^{\delta}-\vol(\overline{A}_{k,\vi,h}\setminus \underline{A}_{k,\vi,h}),$$
		where $\mathcal{I}_k\subseteq \mathcal{K}$ and $\mathcal M_{A,T}$ are as in the proof of Lemma \ref{Discrepancy}.
	For every $0\le i_0,\ldots, i_{r}< M(k)$,	
\textcolor{black}{ $I^{1}_{(\vi)}(g_2^{\lambda})^{-1}$ and $I^{2}_{(\vi)}(g_2^{\lambda})^{-1}$} are contained in 
	$$\left[(-N(t_{k+1})-1)\|(g_2^{\lambda})^{-1}\|,(N(t_{k+1})+1)\|(g_2^{\lambda})^{-1}\|\right]^{r+1}.$$ 
	
	Since $\eta <1$, for sufficiently large $k$, we have that $$(1+\epsilon_{k})t_{k+1}>T_0(N(t_{k+1})+1)\|(g_2^{\lambda})^{-1}\|$$ and $$(1-\epsilon_{k})t_{k}>T_0(N(t_{k+1})+1)\|(g_2^{\lambda})^{-1}\|.$$ 
	
	Hence by \eqref{eq:simple}, it follows that
	\[
	\begin{split}
	\vol(\overline{A}_{k,\vi,h})&=c_{h}|\det(g_2^{\lambda})|^{-1}\prod_{j=0}^{r}\left[(\xi_{k,i_j+1}-\xi_{k,i_j})
	+t_k^{-\kappa_j}\right]
	(1+\epsilon_{k})^{n-r-d}\:t_{k+1}^{n-r-d}\\
	&+O_{\mathcal{K}}\left(\prod_{j=0}^{r}\left[(\xi_{k,i_j+1}-\xi_{k,i_j})
	+t_k^{-\kappa_j}\right]t_{k+1}^{n-r-d-1/2d}N(t_{k+1})^{1/2d}\log(t_{k+1})\right).
	\end{split}
	\]
	
	Since $(1+\epsilon_{k})^{n-r-d}=1+O\left(\frac {1}{k}\right)$, $t_{k+1}^{n-r-d}=k^{\alpha(n-r-d)}\left(1+O_{\kappa,\eta}\left(\frac {1}{k}\right)\right)$ and $N(t)=O(t^{\eta})$,
   \[\begin{split}
   \vol(\overline{A}_{k,\vi,h})&=c_{h}|\det(g_2^{\lambda})|^{-1}k^{\alpha(n-r-d-\kappa)}\\
   &+O_{\mathcal{K},\kappa,\eta}\left(k^{\alpha(n-r-d-\kappa)-1}+k^{\alpha (n-r-d-\kappa-\frac{1}{2d}+\frac{\eta}{2d})} \log k\right).
   \end{split}\]
	
	Since $\alpha>\frac{2d}{1-\eta}$, we get for sufficiently large $k$,
	\[
	\vol(\overline{A}_{k,\vi,h})=c_{h}|\det(g_2^{\lambda})|^{-1}k^{\alpha(n-r-d-\kappa)}+O_{\mathcal{K},\kappa,\eta}\left(k^{\alpha(n-r-d-\kappa)-1} \right).
	\]
	
	By \eqref{eqn:secmom},
	$
	\mu(\mathcal{C}_{k,\vi,h})\ll_{\mathcal{K},\kappa,\eta}k^{(1-2\delta)\alpha (n-r-d-\kappa)}.
	$
	Therefore,
	\[
	\mu( \bigcup\limits_{t_{k}\leq t <t_{k+1}}\mathcal{B}_{t})\leq \sum\limits_{h\in \mathcal{I}_{k}}\sum\limits_{i_0=0}^{M(k)-1}\ldots \sum\limits_{i_{r}=0}^{M(k)-1}\mu(\mathcal{C}_{k,\vi,h})\ll_{\mathcal{K},\kappa,\eta}\frac{k^{d_n}M(k)^{r+1}}{k^{\alpha(n-r-d-\kappa)(2\delta-1)}}.
	\]
	
    Since $M(k)\asymp N(t_{k+1})t_{k+1}^{\beta}$\textcolor{black}{$\le N(t_{k+1})t_{k+1}^{\kappa}$}, we get
    \[
    	\mu( \bigcup\limits_{t_{k}\leq t <t_{k+1}}\mathcal{B}_{t})\ll_{\mathcal{K},\kappa,\eta}\frac{1}{k^{\alpha(n-r-d-\kappa)(2\delta-1)-d_n-\alpha(r+1)(\eta+\kappa)}}.
    \]
    
	Since $\alpha>\frac{d_n+3}{n-r-d-(r+1)\eta-(r+2)\kappa}$ \textcolor{black}{and $\delta >1-\frac{1}{\alpha (n-r-d-\kappa)}$}, we have that $	\alpha(n-r-d-\kappa)(2\delta-1)-d_n-\alpha(r+1)(\eta+\kappa)>1$,
	which implies that $\mu(\overline{\lim_{t\rightarrow \infty}}\mathcal{B}_{t})=0$, thus proving the theorem.
	\end{proof}

We now present Theorem \ref{thm:main2} which leads Theorem \ref{thm:main2_new} directly.

\begin{theorem}\label{thm:main2}
Let $0\leq \eta < \min\{1,\frac{n-r-d}{(r+1)(1+r(r+2))}\}$ and $0\le\kappa< \frac{n-r-d-(r+1)\eta-r(r+1)(r+2)\eta}{(r+1)(r+2)}$. For $0\le j\le r$, let $\kappa_j \ge 0$ be such that $\sum_{j=0}^{r}\kappa_j=\kappa$. Let $N(t)$ be a non-decreasing function such that $N(t)=O(t^{\eta})$. Then there exists $\nu>0 $ such that for almost every $(F,M)\in \mathcal{Y}_{g_2^\lambda}$ and for \textcolor{black}{all $I\subset [-N(t),N(t)]^{r+1}$} of the form $I=I_0\times I_1\times \cdots \times I_{r}$ with $|I_j|\ge t^{-\kappa_j}$ 
\[
\#\{v\in \ZZ^{n} : (F,M)(v)\in I, \|v\|\le t\}=c_{F,M} |I |t^{n-r-d}+O_{F,M}\left(|I|t^{n-r-d-\nu}\right).
\]
\end{theorem}
\begin{proof}
	Let $\kappa\in \Big[0,\frac{n-r-d-(r+1)\eta-r(r+1)(r+2)\eta}{(r+1)(r+2)}\Big)$. Then $\kappa+r\eta<\frac{n-r-d-(r+1)\eta}{(r+1)(r+2)}$. Take $a\in \left(\kappa+r\eta, \frac{n-r-d-(r+1)\eta}{(r+1)(r+2)}\right)$ and let $\kappa'_j=a \text{ for }0\le j\le r$. Then 
\[\kappa < \kappa':=\sum_{j=0}^{r}\kappa'_j=(r+1)a<\frac{n-r-d-(r+1)\eta}{r+2}.\]

	 By Theorem \ref{uniform}, there is $\delta\in(0,1)$ such that for almost every $(F,M)\in \mathcal{Y}_{g_2^{\lambda}}$, there exists $t'_{F,M}>0$ such that for $t>t'_{F,M}$ and for \textcolor{black}{all $I'\subset [-N(t),N(t)]^{r+1}$} of the form $I'=I'_0\times I'_1\times \cdots \times I'_{r}$ with $|I'_j|= t^{-\kappa'_j}$, 
	 \[
	 \left|\mathcal{N}_{F,M}(I',t)-c_{F,M} |I'| t^{n-r-d}\right|<|I'|^{\delta}t^{\delta(n-r-d)}.
	 \]
	 
	Let $\nu=\frac{1}{2}\min\{(1-\delta)(n-r-d-\kappa'),a-(\kappa+r\eta)\}$.
	We first consider the case that there are $n_j\in \NN$, $j=0,\ldots, r$ such that $I=I_0\times \cdots\times I_{r}\subset [-N(t),N(t)]^{r+1}$ for which $|I_j|=n_j t^{-\kappa'_j}$, $0\leq j\leq r$. 
	
	Divide $I_j$ into $n_j$ subintervals each of length $t^{-\kappa'_j}$, 
	i.e. $I_j=\bigsqcup_{l_j=1}^{n_j}I_j^{l_j}$, where $I_j^{l_j}\subseteq I_j$ with $|I_j^{l_j}|=t^{-\kappa'_j}$. 
	This gives a partition of $I$ as $I=\bigsqcup_{l_0=1}^{n_0}\ldots \bigsqcup_{l_{r}=1}^{n_{r}}I_0^{l_0}\times \cdots \times I_{r}^{l_{r}}$ with $|I_0^{l_0}\times\cdots\times I_{r}^{l_{r}}|=t^{-\kappa'}$. 
	Then for $t>t'_{F,M}$,
	 \[
	 \begin{split}
	 &\left|\mathcal{N}_{F,M}(I,t)-c_{F,M}|I| t^{n-r-d}\right|\\
	 	&=\left|\sum_{l_0=1}^{n_0}\ldots \sum_{l_{r}=1}^{n_{r}} \left(\mathcal{N}_{F,M}(I_0^{l_0}\times\cdots\times I_{r}^{l_{r}},t)-c_{F,M}|I_0^{l_0}\times\cdots\times I_{r}^{l_{r}}| t^{n-r-d}\right) \right|\\
	 &\le\sum_{l_0=1}^{n_0}\ldots \sum_{l_{r}=1}^{n_{r}} \left| \mathcal{N}_{F,M}(I_0^{l_0}\times\cdots\times I_{r}^{l_{r}},t)-c_{F,M}|I_0^{l_0}\times\cdots\times I_{r}^{l_{r}}| t^{n-r-d} \right|\\
	 &<\sum_{l_0=1}^{n_0}\ldots \sum_{l_{r}=1}^{n_{r}} t^{-\delta \kappa'}t^{\delta(n-r-d)}= n_0\cdots n_{r}t^{-\delta \kappa'}t^{\delta(n-r-d)}\\
	 &=|I|t^{\kappa'}t^{-\delta \kappa'}t^{\delta(n-r-d)}=|I |t^{n-r-d-(1-\delta)(n-r-d-\kappa')}\\
	 &\le |I|t^{n-r-d-2\nu},
	 \end{split}
	 \]
	 which shows the theorem.
	 
	 \vspace{0.1in}
	 Now, consider the general $I\subset [-N(t),N(t)]^{r+1}$, 
	 which is of the form $I=I_0\times I_1\times \cdots \times I_{r}$ with $|I_j|\ge t^{-\kappa_j}$ for $0\leq j\leq r$. 
	 Then $|I_j|> t^{-\kappa'_j}$ and hence there exists $m_j \in \NN$ such that $m_j t^{-\kappa'_j}<|I_j| \le (m_j +1)t^{-\kappa'_j}$. 
	 Let $I_j^{1}$ and $I_j^{2}$ be such that $|I_j^{1}|=m_j t^{-\kappa'_j}$, $|I_j^{2}|=(m_j+1) t^{-\kappa'_j}$ and $I_j^{1}\subset I_j \subset I_j^{2}$. 
	 Take $\underline{I}=\prod_{j=0}^{r}I_j^{1}$ and $\overline{I}=\prod_{j=0}^{r}I_j^{2}$. 
	 Then $\underline{I}\subset I \subset \overline{I}$. 
	 Since $N(t)=O(t^{\eta})$, we get $m_j=O(t^{\eta+\kappa'_j})$ and hence $|\overline{I}|-|\underline{I}|\le c t^{r\eta-a}$ for some constant $c$. 
	 Let \textcolor{black}{$t_{F,M}=\max\{1,t'_{F,M},(1+c+c_{F,M}c)^{1/\nu}\}$}. By applying the previous result to $\underline{I}$ and $\overline{I}$ \textcolor{black}{and using the estimate $\kappa+r\eta-a<0$}, we obtain that for $t>t_{F,M}$,
	 \[
	 \begin{split}
	 &\left|\mathcal{N}_{F,M}(I,t)-c_{F,M}|I| t^{n-r-d}\right|\\
	 &\le \max\left\{ \left|\mathcal{N}_{F,M}(\underline{I},t)-c_{F,M}|I| t^{n-r-d}\right|,\left|\mathcal{N}_{F,M}(\overline{I},t)-c_{F,M}|I| t^{n-r-d}\right| \right\} \\
	 &\le \max\left\{\left|\mathcal{N}_{F,M}(\underline{I},t)-c_{F,M}|\underline{I}| t^{n-r-d}\right|,\left|\mathcal{N}_{F,M}(\overline{I},t)-c_{F,M}|\overline{I}| t^{n-r-d}\right|\right\} + c_{F,M}(|\overline{I}|-|\underline{I}|)t^{n-r-d}\\
	 &\le |\overline{I}|t^{n-r-d-2\nu}+c_{F,M}(|\overline{I}|-|\underline{I}|)t^{n-r-d}\\
	 &\le |I| t^{n-r-d-2\nu}+ c|I| t^{\kappa+r\eta-a}\:t^{n-r-d-2\nu}+ c_{F,M}c \:|I|t^{\kappa+r\eta-a}\:t^{n-r-d}\\
	 &\le |I| t^{n-r-d-\nu}.
	 \end{split}
	 \]
\end{proof}

\begin{corollary}\label{cor4.3}
	Let \textcolor{black}{$d<n-r$} and $0\leq \eta < \min\{1,\frac{n-r-d}{(r+1)(1+r(r+2))}\}$. Let $N(t)$ be a non-decreasing function such that $N(t)=O(t^{\eta})$ and $\delta(t)$ be a non-increasing function satisfying $\frac{t^{\eta(r+1)(1+r(r+2))-a}}{\delta(t)^{(r+1)^{2}(r+2)}}\rightarrow 0$ for some $a<n-r-d$. Then for almost every $(F,M)\in\mathcal{Y}_{g_{2}^{\lambda}}$ and for sufficiently large $t$,
	\[
	\sup_{\|\xi\|\le N(t)} \min_{v\in \ZZ^{n},\|v\|\le t}\|(F,M)(v)-\xi\|<\delta(t),
	\]
	where $\|\cdot\|$ denotes the supremum norm on $\RR^{n}$.
\end{corollary}

\end{document}